\documentclass{article}
\usepackage[english]{babel}
\usepackage{graphicx}
\usepackage{amssymb,amsmath,amsfonts}

	\newtheorem{theorem}{Theorem}[section]

	\newtheorem{definition}[theorem]{Definition}
	
	\newenvironment{proof}{{\it Proof: }}{$\Box$}

 \newcommand{\Natural}{\mathbb{N}}
 \newcommand{\Integer}{\mathbb{Z}}

 \newcommand{\Complex}{\mathbb{C}}

 \newcommand{\paren}[1]{\left(#1\right)}
 \newcommand{\brac}[1]{\left[#1\right]}
 \newcommand{\set}[1]{\left\{#1\right\}}

\DeclareMathOperator{\sinc}{sinc}

\DeclareMathOperator{\W}{W}
\DeclareMathOperator{\HW}{HW}

\begin{document}

\title{On the enumeration of the roots of arbitrary separable equations using $\HW$ hyper-Lambert maps}

\author{
        Ioannis Galidakis \\
                Department of Mathematics\\
        Agricultural University of Athens\\
        \texttt{jgal@aua.gr}
        \and
        Ioannis Papadoperakis\\
        Department of Mathematics\\
        Agricultural University of Athens\\
        \texttt{papadoperakis@aua.gr}
\\
}
\date{May 2019}

\maketitle
\newpage

\newpage
\section*{Abstract}
In this article we use the $\HW$ maps to solve arbitrary equations $f=0$, by providing an effective enumeration of the roots of $f$, as these project on and at the branches of the $\HW$ maps. This is just an enumeration of the projection points (roots) of a pin-line on the Riemann surface of $f$ through $\HW$.
\section{Introduction}\label{sec1}
The $\HW$ maps have been used to determine the attractors of the infinite exponential whenever it falls into a $p$-cycle in \cite{gal2} and in \cite{gal4} to solve certain transcendental equations such as Kepler's Equation. They have also been used in \cite{nastou} to solve in closed form the generalized Abel differential equation. Here we display a simple algebraic scheme which can be used to utilize the solution of arbitrary equations using the $\HW$ maps, by providing an effective enumeration of all the roots of arbitrary equations $f=0$, using the branches of the maps $\HW$. Imagine an arbitrary multivalued $f$, for which we want force $f(x)=0$. We line-pin the entire Riemann surface of $f$ from top to bottom starting at the complex origin. The local projection pin points $z_i$ will be exactly the roots of $f=0$. Because the branches of $\HW$ can be enumerated starting at the origin, all the roots $z_i$ of $f=0$ can therefore be enumerated and referenced by approximating just an $\epsilon$ pin through the origin.
\section{Definitions}\label{sec31}
Suppose $f_n(z)$ are non-vanishing identically complex functions, with $n\le n_0\in\Natural$. We define $F_{n}(z)\colon\Natural \times\Complex \rightarrow\Complex$ as:

\begin{definition}\label{def31}
\begin{equation*}
F_{n}(z) =\begin{cases}
1& \text{, if $n=1$},\\
e^{f_{n-1}(z)F_{n-1}(z)}& \text{, if $n>1$}.
\end{cases}
\end{equation*}
\end{definition}

\begin{definition}\label{def32}
$G(f_1,f_2,\ldots,f_n;z)=z\cdot F_{n+1}(z)$
\end{definition}

If $n=0$, then $G(z)=z$. If $n=1$ then $G(f_1;z)=ze^{f_1(z)}$. If $n=2$ then $G(f_2,f_1;z)=ze^{f_2e^{f_1(z)}}$. When we write about the $\HW$, we can use the terminology $G(\ldots;z)$, meaning that the corresponding function includes meaningful terms-parameters. The order of the functions is immaterial and we can re-order them to get to the function of interest here, which is the inverse of $G(\ldots;z)$, denoted by,

\begin{equation}\label{eq9}
\HW(f_1,f_2,\ldots,f_n;y)
\end{equation}
In other words $G$ and $\HW$ satisfy the functional relation:
\begin{equation}\label{eq10}
G(\ldots;\HW(\ldots;y))=y
\end{equation}
by supposing always that the list of parameters is identical on both sides. These maps have been called generalized hyper-Lambert $\HW$ functions and in general they are multivalued. We note that when $n=1$, $\HW(y)$ satisfies a more general form which comes from the Lambert function $\W$, i.e., $ze^{f_1(z)}=y$. The Lambert function satisfies $ze^z=y$. The existence of all the $\HW$ is guaranteed in all cases by the Lagrange Inversion Theorem (see \cite[201-202]{syg}).
\section{An indexing scheme for the $\HW$ maps}\label{sec5}
\subsection{An algebraic scheme}\label{sec51}
For the complex maps $\log$ and $\W$, their indexing scheme is the simplest possible, that is $\log(k,z)$ and $\W(k,z)$, $k\in\Integer$. There exists an indexing scheme which indexes identically the mappings $\HW$ but it is not integral. Dubinov in \cite{dub2} solves Kepler's equation, using the following algebraic inversion:

\begin{equation}\label{eq511}
\begin{split}
E-\epsilon\cdot\sin(E)&=M\Rightarrow\\
E\paren{1-\epsilon\frac{\sin(E)}{E}}&=M\Rightarrow\\
E\cdot e^{\log{(1-\epsilon\cdot\sinc{E})}}&=M\Rightarrow\\
E&=\HW\brac{\log(1-\epsilon\cdot\sinc(x));M}
\end{split}
\end{equation}

The inversion above can be generalized producing a removable pole at $z_0$ of multiplicity $n$. Setting $w=(z-z_0)^n$, with $z_0$ such that $f(z_0)=y$, we have:

\begin{equation}\label{eq512}
\begin{split}
f(z)&=y\Rightarrow\\
(z-z_0)^n\cdot\frac{f(z)}{(z-z_0)^n}&=y\Rightarrow\\
w\cdot e^{\log\paren{\frac{f(z)}{w}}}&=y\Rightarrow\\
w&=\HW\brac{\log\paren{\frac{f(z)}{w}};y}\Rightarrow\\
(z-z_0)^n&=\HW\brac{\log\paren{\frac{f(z)}{(z-z_0)^n}};y}\Rightarrow\\
z&=\HW\brac{\log\paren{\frac{f(z)}{(z-z_0)^n}};y}^{\frac{1}{n}}+z_0
\end{split}
\end{equation}

The scheme above gives an index into the set of the $\HW$ functions, in the form of a functional parameter as $\log\paren{\frac{f(z)}{(z-z_0)^n}}$. Now, if we know $f(z)$, this scheme can give identities which must hold identifying this way the corresponding function.

We can now list how the most important categories of complex functions are solved based on this index.
\subsection{Polynomial functions}\label{sec52}
Suppose then that $f(z)=\prod\limits_{k=1}^N(z-z_k)^{n_k}$. Keeping $k$ fixed and setting $w=(z-z_k)^{n_k}$, we have,

\begin{equation}\label{eq521}
\begin{split}
w&=\HW\brac{\log\paren{\frac{f(z)}{w}};y}\Rightarrow\\
z&=\HW\brac{\log\paren{\frac{f(z)}{(z-z_k)^{n_k}}};y}^{\frac{1}{n_k}}+z_k\Rightarrow\\
z&=\HW\brac{\log(f(z))-n_k\log\paren{z-z_k};y}^{\frac{1}{n_k}}+z_k
\end{split}
\end{equation}

\begin{theorem}\label{the521} If $f(z)=\prod\limits_{k=1}^N(z-z_k)^{n_k}$ is a complex polynomial function, then the inverse of $f(z)$ relative to $y$ is given by the function $\HW$ and the last equation of \eqref{eq521}, whose Riemann surface has at most $m=\sum\limits_{k=1}^N n_k$ branches, indexed by $m$, with $k\in\Natural$.
\end{theorem}
\begin{proof} The last expression of \eqref{eq521} is true for any $k\in\set{1,2,\ldots,N}$, therefore the multiplicity is at least $m$ because for each $k$ the multiplicity is at least $n_k$ and each $n_k$ may give different branches. This means that the expression can index fully all the branches of the corresponding $\HW$ using only an integral index $k$.
\end{proof}

\begin{theorem}\label{the522} If $f(z)$ is a complex polynomial function, the roots of $f(z)=y$ are given directly by a suitable $\HW$ function.
\end{theorem}
\begin{proof} Using equation \eqref{eq9} of Definition \eqref{def32}, follows that for each $\HW$, $\HW(\ldots;0)=0$, therefore calculating the corresponding $\HW$ of the last equations in \eqref{eq521} at $y=0$, forces $z=z_k$ and these are the roots of $f(z)=y$. Therefore, we can extract all the roots of equation $f(z)=y$, manually. The first root, suppose $z_1$, is extracted as,

\begin{equation*}
\begin{split}
z_1&=\HW\brac{\log\paren{\frac{f(z)}{z}};y}\\
g_1(z)&=\frac{f(z)-y}{z-z_1}\\
\end{split}
\end{equation*}

Having the root $z_1$, the rest of the roots can be extracted recursively for $1\le k\le N-1$ as,

\begin{equation*}
\begin{split}
z_{k+1}&=\lim_{\epsilon\to 0^+}\HW\brac{\log\paren{\frac{g_k(z)}{z}};\epsilon}\\
g_{k+1}(z)&=\frac{g_k(z)}{z-z_{k+1}}\\
\end{split}
\end{equation*}
and the Theorem follows.
\end{proof}
\subsection{Rational functions}\label{sec53}
We suppose that $f(z)=P(z)/Q(z)$, with $P(z)$, $Q(z)$ polynomial functions. We have similar results here.

\begin{theorem}\label{the531} If $f(z)=P(z)/Q(z)$ is a complex rational function such that $N=\max\set{\deg(P),\deg(Q)}$, then the inverse of $f(z)$ relative to $y$ is given by:

\begin{equation*}
z=\HW\brac{\log\paren{\frac{P(z)-y\cdot Q(z)}{(z-z_k)^{n_k}}};y}^{\frac{1}{n_k}}+z_k
\end{equation*}

whose Riemann surface has at most $m=\sum\limits_{k=1}^N n_k$ branches, indexed by $m$, with $k\in\Natural$.
\end{theorem}
\begin{proof} If $F(z)=P(z)-y\cdot Q(z)$, then $F(z)$ is a polynomial of degree $N$, in which case the Theorem follows similarly, with $f(z)$ replaced by $F(z)$.
\end{proof}
\begin{theorem}\label{the532} If $f(z)$ is a complex rational function, the roots of $f(z)=y$ can be given by a suitable $\HW$ function.
\end{theorem}
\begin{proof} Similarly, if $F(z)=P(z)-y\cdot Q(z)$, then $F(z)$ is polynomial map of degree $N$, therefore we can extract its roots as:

\begin{equation*}
\begin{split}
z_1&=\lim_{\epsilon\to 0^+}\HW\brac{\log\paren{\frac{F(z)}{z}};\epsilon}\\
g_1(z)&=\frac{F(z)}{z-z_1}\\
\end{split}
\end{equation*}

Having $z_1$, the rest of the roots can be extracted recursively for $1\le k\le N-1$ as,
\begin{equation*}
\begin{split}
z_{k+1}&=\lim_{\epsilon\to 0^+}\HW\brac{\log\paren{\frac{g_k(z)}{z}};\epsilon}\\
g_{k+1}(z)&=\frac{g_k(z)}{z-z_{k+1}}\\
\end{split}
\end{equation*}
and the Theorem follows.
\end{proof}

We observe that when $Q(z)=1$, the case of a polynomial function arises.
\subsection{Analytic functions}\label{sec54}
For an analytic function $f(z)=\sum\limits_{n=0}^\infty \alpha_n\cdot(z-z_0)^n$ in some region $D\subseteq\Complex$, with $z_0\in D$, we have similar results.

\begin{theorem}\label{the541} If $f(z)=\sum\limits_{n=0}^\infty \alpha_n\cdot(z-z_0)^n$ is a complex analytic function, then the inverse of $f(z)$ relative to $y$ is given by a suitable $\HW$ function:

\begin{equation*}
z=\HW\brac{\log\paren{\frac{f(z)}{(z-z_k)^{n_k}}};y}^{\frac{1}{n_k}}+z_k
\end{equation*}

whose Riemann surface has infinitely many branches given by $n\in\Natural$.
\end{theorem}
\begin{proof} Suppose $T_N(z)=\sum\limits_{n=0}^N\alpha_n\cdot(z-z_0)^n$, is the corresponding Taylor polynomial of degree $N$. Then $T_N(z)$ is obviously a polynomial function, therefore the inverse of $T_N(z)$ relative to $y$ is given again by Theorem \eqref{the521}.

\begin{equation}\label{eq541}
z=\HW\brac{\log\paren{\frac{T_N(z)}{(z-z_k)^{n_k}}};y}^{\frac{1}{n_k}}+z_k
\end{equation}

$T_N(z)\to f(z)$ uniformly in compact subsets and the $\HW$ are analytic (\cite{gal2}), therefore \eqref{eq541} implies that the inverse is given by:

\begin{equation}\label{eq542}
\begin{split}
z&=\lim_{N\to\infty}\HW\brac{\log\paren{\frac{T_N(z)}{(z-z_k)^{n_k}}};y}^{\frac{1}{n_k}}+z_k\Rightarrow\\
z&=\HW\brac{\log\paren{\frac{\lim\limits_{N\to\infty}T_N(z)}{(z-z_k)^{n_k}}};y}^{\frac{1}{n_k}}+z_k\Rightarrow\\
z&=\HW\brac{\log\paren{\frac{f(z)}{(z-z_k)^{n_k}}};y}^{\frac{1}{n_k}}+z_k\\
\end{split}
\end{equation}
and the Theorem follows.
\end{proof}

We observe that in this case the inverse function has infinitely many branches, since $N$ is not bounded.

\begin{theorem}\label{the542} If $f(z)$ is a complex analytic function, the roots of $f(z)=y$ are given again by a $\HW$ function.
\end{theorem}
\begin{proof} We can extract the roots as:

\begin{equation}\label{eq543}
\begin{split}
z_1&=\HW\brac{\log\paren{\frac{f(z)}{z}};y}\\
g_1(z)&=\frac{f(z)-y}{z-z_1}\\
\end{split}
\end{equation}

The rest of the roots can be again extracted recursively for $1\le k$ as,

\begin{equation}\label{eq544}
\begin{split}
z_{k+1}&=\lim_{\epsilon\to 0^+}\HW\brac{\log\paren{\frac{g_k(z)}{z}};\epsilon}\\
g_{k+1}(z)&=\frac{g_k(z)}{z-z_{k+1}}\\
\end{split}
\end{equation}
and the Theorem follows.
\end{proof}
\section{$\HW$ functional index}\label{sec55}
An open problem set in \cite[1114-1115]{gal4} is whether there is a way to effectively index the numbering of the branches of the $\HW$ functions. With the following Theorem we show that the answer is affirmative.

\begin{theorem}\label{the551} If $f(z)$ is a complex function and $z_k\in\Complex$, $k\in\Natural$, such that $f(z_k)=y$ and suppose $g_k(z)$ follows as in equations \eqref{eq543} - \eqref{eq544}. Then, if $\HW$ is the inverse of $f(z)$ relative to $y$, the following scheme covers all the branches of this inverse of $f(z)$:

\begin{equation*}
z_{k+1}=\begin{cases}
\HW\brac{(0,)\log\paren{\frac{f(z)}{z}};y}\text{, if $k=0$},\\
\lim\limits_{\epsilon\to 0^+}\HW\brac{k,\log\paren{\frac{g_k(z)}{z}};\epsilon}\text{, if $k>0$}.
\end{cases}
\end{equation*}
\end{theorem}
\begin{proof} The proof follows from (3) along with Theorems 3.1, 3.3 and 3.5. Note that for a specific analytic $f$ expanded around $z_k$, we define $F(z)=\log\paren{\frac{f(z)}{(z-z_k)^{n_k}}}$. The map $F$ creates a Laurent series with residue $\exp(a_{n_k})$, which is gotten from the $\HW$ through the Residue Theorem of Cauchy for $f$, with winding number $a_{n_k}$ around $z_k$. Consequently, the repeated application of $F$ (via $g_k$) for $z=z_k$, extracts recursively all the roots $z_k$ of the inverse and as such it can be used as an index for the corresponding Riemann surface.
\end{proof}
\section{Conclusions}\label{sec56}
The $\HW$ maps can solve any equation $f=0$, provided it can be brought into a separable form with all $z$'s on the left and one $w$ on the right. Further, the enumeration of the roots is origin consistent relative to the force of $f$.
\section{Appendix: programming with the $\HW$ maps}\label{secpar2}
Code for the $\HW$ maps is given below. Arguments are HW(functional index,y,n)
\begin{verbatim}
restart;
Digits:=40;
HW := proc ()
local y, n, c, s, p, sol, i, aprx, dy, dist, r, newr, oldr, fun, dfun,eps;
if nargs < 2 then ERROR("At least two arguments required") end if;
n := args[-1]; y := args[-2]; c := [args[1 .. -3]];
if y = 0 then 0 else dist := infinity;
eps:=1e-10; fun := 1; for i from 1 to nargs-2 do fun := exp(c[-i]*fun) end do;
fun := z*fun-y; dfun := diff(fun, z);
s := series(fun, z, n); p := convert(s, polynom);
sol := {fsolve(p = 0, z, complex)};
for i from 1 to nops(sol) do aprx := evalf(subs(z = op(i, sol), fun));
dy := evalf(abs(aprx)); if dy <= dist then r := op(i, sol);
dist := dy end if end do; oldr := r;
newr := r-evalf(subs(z = r, fun)/subs(z = r, dfun));
for i from 1 to 1000 while abs((oldr-newr)/oldr)>eps do oldr := newr;
newr := newr-evalf(subs(z = newr, fun)/subs(z = newr, dfun)) end do;
newr end if end proc:
\end{verbatim}

\emph{Example: } Using the program with five decimal digits accuracy to solve the equation $(z-2)(z-3)(z-5)=2$,
\begin{verbatim}
y:=2;
f:=z->(z-2)*(z-3)*(z-5);
z1:=HW(log(f(z)/z),y,10);
g1:=z->(f(z)-y)/(z-z1);
z2:=HW(log(g1(z)/z),1e-20,10);
g2:=z->g1(z)/(z-z2);
z3:=HW(log(g2(z)/z),1e-20,10);
\end{verbatim}

gives:

\begin{equation*}
\begin{split}
z_1&\simeq 2.36523-0.69160i\\
z_2&\simeq 2.36523+0.69160i\\
z_3&\simeq 5.26953\\
\end{split}
\end{equation*}

Using the program for an approximate solution with Maple,

\begin{verbatim}
solve(f(z)=y,z);
evalf(%);
\end{verbatim}

gives:

5.26953, 2.36523+0.69160i, 2.36523-0.69160i.

\emph{Example 2: } Using the program to five digits of accuracy to solve the equation $(z-2)(z-3)/(z-5)/(z-1)=2$,
\begin{verbatim}
y:=2;
f:=z->(z-2)*(z-3)/(z-5)/(z-1);
P:=unapply(numer(f(z)),z);
Q:=unapply(denom(f(z)),z);
F:=unapply(P(z)-y*Q(z),z);
z1:=HW(log(F(z)/z),1e-10,10);
g1:=z->F(z)/(z-z1);
z2:=HW(log(g1(z)/z),1e-10,10);
\end{verbatim}

gives:

\begin{equation*}
\begin{split}
z_1&\simeq 6.37228\\
z_2&\simeq 0.62771\\
\end{split}
\end{equation*}

Using Maple approximation code,

\begin{verbatim}
solve(F(z)=y,z);
evalf(%);
\end{verbatim}

gives:

6.37228, 0.62771.

\emph{Example 3: } Using the program to five decimals of accuracy to solve the equation $\sin(z)=1/2$,
\begin{verbatim}
y:=1/2;
f:=z->sin(z);
z1:=HW(log(f(z)/z),y,10);
g1:=z->(f(z)-y)/(z-z1);
z2:=HW(log(g1(z)/z),1e-20,10);
g2:=z->g1(z)/(z-z2);
z3:=HW(log(g2(z)/z),1e-20,10);
\end{verbatim}

gives:

\begin{equation*}
\begin{split}
z_1&\simeq 0.52359\\
z_2&\simeq 2.61799\\
z_3&\simeq -3.66519\\
z_4&\simeq\cdots\\
\end{split}
\end{equation*}

The results are approximations of the numbers $\pi/6$, $5\pi/6$, $-7\pi/6$,..., which are the roots of $\sin(z)=1/2$. Many more complex equations can be solved here, provided they are separable and the terms are analytic, like 

\emph{Example 5:} Using the code with five decimal accuracy to solve the equation $\sin(z)+\exp(\sin(z))/\sqrt{1+\tanh(z)}$,
\begin{verbatim}
y=1/2;
f:=sin(z)+exp(sin(z))/sqrt(1+tanh(z));
z1:=HW(log(f(z)/z), y, 10);
g1 := z->(f(z)-y)/(z-z1);
z2:=HW(log(f1(z)/z),1e-20,10);
g2:=z->g1(z)/(z-z2);
z3:=HW(log(g2(z)/z),1e-20,10);
g3:=z->g2(z)/(z-z3);
z4:=HW(log(g3(z)/z), 0.1e-19, 10);
\end{verbatim}

gives:
\begin{equation*}
\begin{split}
z_1&\simeq -0.37435\\
z_2&\simeq -1.71811\\
z_3&\simeq 3.26659\\
z_4&\simeq 6.15846\\
\end{split}
\end{equation*}

While using Maple approximation code,

\begin{verbatim}
solve(f(z)=y,z);
\end{verbatim}

gives an open answer in terms of ``RootOf'', i.e. it cannot relay the roots directly.

\emph{Example 6:} Using the code with five decimal accuracy to solve the equation $z^3-4z^2+5z$,
\begin{verbatim}
y:=2;
f:=z->z^3-4*z^2+5*z;
z1:=HW(log(f(z)/z),y,10);
g1:=z->(f(z)-y)/(z-z1);
z2:=HW(log(f1(z)/z),1e-20,10);
g2:=z->g1(z)/(z-z2);
z3:=HW(log(g2(z)/z),1e-20,10);
\end{verbatim}

gives:

\begin{equation*}
\begin{split}
z_1&\simeq 2\\
z_2&\simeq 1.00005\\
z_3&\simeq 1.00000\\
\end{split}
\end{equation*}

The description calculates correctly roots with multiplicity greater than 1.

The example is $f(z)-y=(z-1)^2(z-2)$, therefore the multiplicity of the root 1 is indeed 2.

\newpage
\bibliographystyle{plain}
\bibliography{Thesis}

\end{document}